\documentclass[11pt]{amsart}
\usepackage{amsmath, amssymb, amsfonts, verbatim, graphicx, amsthm, tikz-cd, mathrsfs, color, comment, fancyhdr, amsrefs}
\usepackage[top=1in, bottom=1in, left=1in, right=1in]{geometry}
\usepackage{enumerate}
\usepackage{enumitem}
\usepackage{url}

\usepackage[foot]{amsaddr}

\allowdisplaybreaks[1]

\usepackage[sc]{mathpazo}
\usepackage[T1]{fontenc}

\newcommand{\vertiii}[1]{{\left\vert\kern-0.25ex\left\vert\kern-0.25ex\left\vert #1
		\right\vert\kern-0.25ex\right\vert\kern-0.25ex\right\vert}}

\theoremstyle{plain}
\newtheorem{Thm}{Theorem}[section]
\newtheorem{Prop}{Proposition}[section]
\newtheorem{Lem}{Lemma}[section]
\newtheorem{Cor}{Corollary}[section]
\newtheorem{Fact}{Fact}[section]

\theoremstyle{definition}
\newtheorem{Def}{Definition}[section]

\theoremstyle{remark}
\newtheorem{Rmk}{Remark}[section]

\renewcommand{\epsilon}{\varepsilon}

\usepackage{setspace}

\usepackage{hyperref}

\title{Non-Autonomous Spatial-Temporal Differentiation Theorems for Group Endomorphisms}

\author{I. Assani$^1$, A. Young$^2$}
\address{University of North Carolina at Chapel Hill}

\email{$^1$assani@email.unc.edu, $^2$aidanjy@live.unc.edu}
\begin{document}
	
	\maketitle
	
	\begin{abstract}
	We introduce a non-autonomous generalization of spatial-temporal differentiations, and prove results about probabilistically and topologically generic behaviors of certain spatial-temporal differentations generated by endomorphisms of compact abelian metrizable groups.
	\end{abstract}

	In \cite{StartOfAssani-Young}, we introduced the notion of a spatial-temporal differentiation problem. Here, we introduce a generalization of this concept to the setting of non-autonomous dynamical systems, and prove probabilistic and topological results about certain random spatial-temporal differentiations on compact abelian metrizable groups.
	
	This paper is organized as follows:
	\begin{itemize}
		\item In Section \ref{Intro}, we provide a definition of non-autonomous dynamical systems for our purposes. We also describe what a spatial-temporal differentiation problem would look like in this non-autonomous setting.
		\item In Section \ref{Harmonic analysis}, we introduce the notion of uniform distribution, and describe how uniform distribution in a compact group is related to the representation theory of that group. We end in proving a metric result about the uniform distribution of the trajectory of a point under a sequence of group endomorphisms under the hypothesis that the group endomorphisms satisfy a property we call the \emph{Difference Property}.
		\item In Section \ref{Difference Property}, we consider questions about when the Difference Property makes the group endomorphisms surjective, and whether a sequence with the Difference Property can exist on a given group.
		\item In Section \ref{Probabilistic results}, we prove a probabilistic result about non-autonomous spatial-temporal differentiations relative to a sequence of group endomorphisms with the Difference Property and a sequence of concentric balls with rapidly decaying radii, demonstrating that the set of $x \in G$ which generate well-behaved spatial-temporal differentiations is of full measure.
		\item In Section \ref{Not concentric}, we prove a probabilistic result about uniformly distributed sequences of the form $\left( T_n \cdots T_1 g_{\Lambda_n} \right)_{n = 0}^\infty$, where $(g_n)_{n = 0}^\infty \in G^{\mathbb{N}_0}$ and $(\Lambda_n)_{n = 1}^\infty$ is an increasing sequence of natural numbers.
		\item In Section \ref{Topological results}, we prove a topological counterpoint to Theorem \ref{Main Theorem}, demonstrating that the set of $x \in G$ which generate pathological spatial-temporal differentiations is comeager.
	\end{itemize}

	We thank the referee for their careful reading of this paper.
	
	\section{Introducing non-autonomous dynamical systems}\label{Intro}
	
	Our definition of a non-autonomous dynamical system is inspired by the "process formulation" found in \cite{Kloeden}, though adapted for our ergodic-theoretic purposes. We state the definition in excess generality, because it is more important to us that the definition capture the \emph{concept} of non-autonomy; the study of (autonomous) dynamical systems comes in many diverse flavors, so we would like our definition of non-autonomous dynamical systems to reflect that diversity.
	
	Let $\mathbb{N}_0 = \mathbb{Z} \cap [0, \infty)$ be the semigroup of nonnegative integers (distinguished from the set $\mathbb{N}$ of strictly positive integers), and let $X$ be an object in a category $\mathcal{C}$. Let $\operatorname{Hom}_{\mathcal{C}}(X, X)$ denote the semigroup of endomorphisms on $X$ in the category $\mathcal{C}$. A \emph{non-autonomous dynamical system} is a pair $(X, \tau)$, where $\tau$ is a family of maps $\{ \tau(s, t) \in \operatorname{Hom}_{\mathcal{C}}(X, X) \}_{s, t \in \mathbb{N}_0, s \geq t}$ satisfying the following conditions.
	\begin{enumerate}
		\item $\tau(s, s) = \operatorname{id}_{X}$ for all $s \in \mathbb{N}_0$
		\item $\tau(s, u) = \tau(s, t) \tau(t, u)$ for all $s, t, u \in \mathbb{N}_0, s \geq t \geq u$,
	\end{enumerate}
	where the composition of endomorphisms of $X$ is abbreviated as multiplication. We refer to $\tau$ as the \emph{process}.
	
	The essential difference between an autonomous and a non-autonomous system is that the transition map $\tau(s, t)$ is dependent on both the "starting time" $t$ and the "ending time" $s$. The system would be autonomous if it had the additional property that $\tau(s, t) = \tau(s - t, 0)$ for all $s , t \in \mathbb{N}_0 , s \geq t$, indicating that the transition map depends only on the elapsed time between $t$ and $s$.
	
	Similar to how an autonomous dynamical system can be treated in terms of either an action of $\mathbb{N}_0$ on a phase space, or equivalently in terms of its generating transformation $T$, a non-autonomous dynamical system as we have formulated it above can be understood in terms of a family of generators $T_t = \tau(t, t - 1), t \in \mathbb{N}$. Likewise, a family of generators $\{T_t \in \operatorname{Hom}_{\mathcal{C}}(X, X) \}_{t \in \mathbb{N}}$ can be understood as generating a non-autonomous dynamical system by
	\begin{align*}
		\tau(s, t)	& = \tau(s, s - 1) \tau(s - 1, s - 2) \cdots \tau(t + 1, t) \\
		& = T_s T_{s - 1} \cdots T_{t + 1} .
	\end{align*}
	The approaches are equivalent, but we will typically be approaching these non-autonomous systems from the perspective of starting with the generators $(T_n)_{n = 1}^\infty$ and building $\tau(\cdot, \cdot)$ from that sequence.
	
	For our purposes, that category $\mathcal{C}$ will be the category whose objects are compact topological spaces $X$ endowed with Borel probability measures $\mu$, and whose morphisms are continuous maps. We will not in general assume these maps are measure-preserving. Though the assumption that maps are measure-preserving is typically vital in the autonomous setting, we will eventually be considering situations where interesting results are possible without the explicit assumption that the maps in question are measure-preserving. For measurable sets $F$ with $\mu(F) > 0$, set $\alpha_F(f) = \frac{1}{\mu(F)} \int_F f \mathrm{d} \mu$. We are interested in questions of the following forms:
	\begin{itemize}
		\item Let $(F_k)_{k = 1}^\infty$ be a sequence of Borel subsets of $X$ for which $\mu(F_k) > 0$, and let $f \in L^\infty(X, \mu)$. Then what can be said about the limiting behavior of
		$$\left( \alpha_{F_k} \left( \frac{1}{k} \sum_{i = 0}^{k - 1} T_i T_{i - 1} \cdots T_1 f \right) \right)_{k = 1}^\infty ?$$
		\item Suppose the $F_k = F_k(x)$ are "indexed" by $x \in X$. Then can we make any probabilistic claims about the generic behavior of the sequence $\left( \alpha_{F_k(x)} \left( \frac{1}{k} \sum_{i = 0}^{k - 1} T_i T_{i - 1} \cdots T_1 f \right) \right)_{k = 1}^\infty$?
		\item Under the same conditions, can we make any topological claims about the generic behavior of the sequence $\left( \alpha_{F_k(x)} \left( \frac{1}{k} \sum_{i = 0}^{k - 1} T_i T_{i - 1} \cdots T_1 f \right) \right)_{k = 1}^\infty$?
	\end{itemize}
	Theorem \ref{Main Theorem} is a result of the second type, describing the probabilistically generic behavior of a spatial-temporal differentiation along the sequence $B_k(x)$, where $B_k(x)$ is a ball centered at $x$ with radius decaying rapidly to $0$. Theorems \ref{Mance Theorem} and \ref{Stronger Mance Theorem} are of the third type, describing the topologically generic behavior of a spatial-temporal differentiation along the sequence $B_k(x)$.
	
	\section{Uniform distribution and harmonic analysis}\label{Harmonic analysis}
	
	Before proceeding, we define the notion of uniform distribution. The study of uniformly distributed sequences began with Weyl's investigation of "uniform distribution modulo $1$", expanding on Kronecker's Theorem in Diophantine approximation \cite{WeylUD}. This notion was then extended by Hlawka to apply to compact probability spaces \cite{HlawkaUD}. The study of uniform distribution in compact groups in particular was first initiated by Eckmann in \cite{Eckmann}; though Eckmann's initial definition of uniform distribution for compact groups contained a significant error, the initial paper still contained several foundational results in the theory, including the Weyl Criterion for Compact Groups (Proposition \ref{Weyl Criterion}). For a more through history of the topic, the reader is referred to the note at the end of 4.1 in \cite{KN}.
	
	\begin{Def}
		Let $X$ be a compact Hasudorff topological space endowed with a regular Borel probability measure $\mu$. A sequence $(x_n)_{n = 0}^\infty$ in $X$ is called \emph{uniformly distributed} with respect to the measure $\mu$ if
		$$\frac{1}{k} \sum_{i = 0}^{k - 1} f(x_i) \stackrel{k \to \infty}{\to} \int f \mathrm{d} \mu$$
		for all $f \in C(X)$.
	\end{Def}
	
	Let $G$ denote a compact topological group with identity element $1$ and Haar probability measure $\mu$. Throughout this section, we will be dealing only with the dynamics of compact groups, so $G$ will always denote a compact group, $\mu$ will always refer to the Haar probability measure on the compact group $G$, and $\operatorname{Bo}(X)$ will always refer to the Borel $\sigma$-algebra on a topological space $X$. We  will also write $L^p(G) : = L^p(G, \mu)$, taking the measure $\mu$ to be understood.
	
	When it comes to topological groups, the uniform distribution of topological groups can be characterized in terms of the representation theory of the group. We review here some important concepts from the representation theory of topological groups so that we may state this relation. Our brisk summary of the basic representation theory of compact groups mostly follows \cite{FollandHarmonic}.
	
	Let $\mathbb{U}(\mathcal{H})$ denote the group of unitary operators on a Hilbert space $\mathcal{H}$, where we endow $\mathbb{U}(\mathcal{H})$ with the strong operator topology. A \emph{unitary representation of $G$ on $\mathcal{H}$} is a continuous group homomorphism $\pi : G \to \mathbb{U}(\mathcal{H})$. Though non-unitary representations exist, we will be dealing here exclusively with unitary representations, so we do not bother to define non-unitary representations.
	
	We call a closed subspace $\mathcal{M}$ of $\mathcal{H}$ an \emph{invariant subspace} for a unitary representation $\pi : G \to \mathbb{U}(\mathcal{H})$ if $\pi(x) \mathcal{M} \subseteq \mathcal{M}$ for all $x \in G$. Since $\pi(x)|_{\mathcal{M}} : \mathcal{M} \to \mathcal{M}$ is unitary on $\mathcal{M}$, we call $\pi^{\mathcal{M}} : x \mapsto \pi(x)|_{\mathcal{M}} \in \mathbb{U}(\mathcal{M})$ the \emph{subrepresentation} of $\pi$ corresponding to $\mathcal{M}$. A unitary representation $\pi : G \to \mathbb{U}(\mathcal{H})$ is called \emph{irreducible} if its only invariant subspaces are $\mathcal{H}$ and $\{0\}$. Two unitary representations $\pi_1 : G \to \mathbb{U}(\mathcal{H}_1), \pi_2 : G \to \mathbb{U}(\mathcal{H}_2)$ are called \emph{unitarily equivalent} if there exists a unitary map $U : \mathcal{H}_1 \to \mathcal{H}_2$ such that $\pi_2(x) = U \pi_1(x) U^{-1}$ for all $x \in G$. We denote by $[\pi]$ the unitary equivalence class of a representation $\pi : G \to \mathbb{U}(\mathcal{H})$.
	
	\begin{Fact}\label{Compact group reps}
		If $G$ is a compact group, then for every irreducible unitary representation $\pi : G \to \mathbb{U}(\mathcal{H})$, the space $\mathcal{H}$ is of finite dimension.
	\end{Fact}
	
	\begin{proof}
	\cite[Theorem 5.2]{FollandHarmonic}.
	\end{proof}
	
	Given a unitary representation $\pi : G \to \mathbb{U}(\mathcal{H})$, we define the \emph{dimension} $\dim \pi$ of $\pi$ to be $\dim (\mathcal{H})$; since a unitary equivalence of representations induces a unitary isometry between the spaces they act on, we can conclude that the dimension of a representation is invariant under unitary equivalence. We use $\hat{G}$ to denote the family of unitary equivalence classes of irreducible unitary representations of $G$. We note that this notation is consistent with the use of $\hat{G}$ to refer to the Pontryagin dual of a locally compact abelian group $G$, since when $G$ is locally compact abelian, the irreducible unitary representations are exactly the continuous homomorphisms $G \to \mathbb{S}^1 \subseteq \mathbb{C}$.
	
	In particular, there will always exist at least one irreducible representation of dimension $1$, specifically the map $x \mapsto 1 \in \mathbb{C}$. We call this the \emph{trivial representation} of $G$.
	
	Given an irreducible unitary representation $\pi : G \to \mathbb{U}(\mathcal{H})$, we define the \emph{matrix elements} of $\pi$ to be the functions $G \to \mathbb{C}$ given by
	\begin{align*}x & \mapsto \left< \pi(x) u, v \right> & (u, v \in \mathcal{H}) . \end{align*}
	Because $\pi$ is continuous with respect to the strong operator topology on $\mathcal{H}$, it follows that the matrix elements are continuous. Given an orthonormal basis $\{ e_i \}_{i = 1}^n$ of $\mathcal{H}$, we define the functions $\{ \pi_{i, j} \}_{i, j = 1}^n$ by
	$$\pi_{i, j}(x) = \left< \pi(x) e_i , e_j \right> .$$
	These $\pi_{i, j}$ in fact define matrix entries for $\pi$ in the basis $\{ e_{i} \}_{i = 1}^n$.
	
	In particular, the trivial representation will have constant matrix elements.
	
	This is sufficient framework to state the results we will be drawing on.
	\begin{Fact}\label{Peter-Weyl}[Peter-Weyl Theorem]
		Let $G$ be a compact group, and let $V \subseteq C(G)$ be the subspace of $C(G)$ spanned by
		$$\left\{ \pi_{p , q} : p , q = 1, \ldots, \dim \pi; [\pi] \in \hat{G} \right\} .$$
		Then $V$ is dense in $C(G)$ with respect to the uniform norm. Furthermore,
		$$\left\{ \sqrt{\dim \pi} \pi_{p , q} : p , q = 1, \ldots, \dim \pi; [\pi] \in \hat{G} \right\} $$
		is an orthonormal basis for $L^2(G)$.
	\end{Fact}
	
	\begin{proof}
	\cite[Theorem 5.12]{FollandHarmonic}.
	\end{proof}
	
	\begin{Cor}\label{Schur}
		If $\pi : G \to \mathbb{U}(\mathcal{H})$ is a nontrivial irreducible unitary representation, then $\int \pi_{p, q} \mathrm{d} \mu = 0$ for all $(p, q) \in \{1, 2, \ldots, \dim \pi\}^2$.
	\end{Cor}
	
	\begin{proof}
		$$\int \pi_{p, q} \mathrm{d} \mu = \int \pi_{p, q} \cdot 1 \mathrm{d} \mu = \left< \pi_{p, q}, 1 \right>_{L^2(G)} = 0,$$
		since the constant function $1$ is a normalized matrix term of the trivial representation.
	\end{proof}
	
	Finally, we return to the subject of uniform distribution.
	
	\begin{Prop}[Weyl Criterion for Compact Groups]\label{Weyl Criterion}
		Let $G$ be a compact group, and $(x_n)_{n = 0}^\infty$ a sequence in $G$. Then the following are equivalent.
		\begin{enumerate}
			\item The sequence $(x_n)_{n = 0}^\infty$ is uniformly distributed in $G$.
			\item For all nontrivial irreducible unitary representations $\pi : G \to \mathbb{U}(\mathcal{H})$, and all $(p, q) \in \{1, 2, \ldots, \dim \pi\}^2$ we have
			$$\frac{1}{k} \sum_{i = 0}^{k - 1} \pi_{p, q}(x_i) \stackrel{k \to \infty}{\to} 0 .$$
			\item For all nontrivial irreducible unitary representations $\pi : G \to \mathbb{U}(\mathcal{H})$, we have
			$$\frac{1}{k} \sum_{i = 0}^{k - 1} \pi(x_i) \stackrel{k \to \infty}{\to} \mathbf{0} ,$$
			where $\mathbf{0}$ denotes the zero operator on $\mathcal{H}$, and the convergence is meant in the operator norm.
		\end{enumerate}
	\end{Prop}
	
	\begin{proof}
	\cite[Chapter 4, Theorem 1.3]{KN}.
	\end{proof}

	\begin{Lem}\label{Group homomorphisms preserve uniform distribution}
	Let $\varphi : G \twoheadrightarrow G_1$ be a continuous surjective group homomorphism, and let $(x_n)_{n = 1}^\infty$ be a uniformly distributed sequence in $G$. Then $\left( \varphi(x_n) \right)_{n = 1}^\infty$ is uniformly distributed in $G_1$.
	\end{Lem}

	\begin{proof}
	Let $\pi : G_1 \to \mathbb{U}(\mathcal{H})$ be a nontrivial unitary representation of $G_1$. Then $\pi \circ \varphi$ is a nontrivial unitary representation of $G$, since $\varphi$ is surjective. Therefore
	\begin{align*}
	\frac{1}{k} \sum_{i = 0}^{k - 1} \pi(\varphi(x_i))	& = \frac{1}{k} \sum_{i = 0}^{k - 1} (\pi \circ \varphi)(x_i)	& \stackrel{k \to \infty}{\to} \mathbf{0} .
	\end{align*}
	We can thus apply Proposition \ref{Weyl Criterion}.
	\end{proof}

	\begin{Rmk}
	Lemma \ref{Group homomorphisms preserve uniform distribution} is listed as Theorem 1.6 in Chapter 4 of \cite{KN}. We include the proof here for the sake of self-containment. Lemma \ref{Group homomorphisms preserve uniform distribution} will be important when proving Theorem \ref{Shifts and uniform distribution}.
	\end{Rmk}
	
	We will be interested specifically in the case where the Weyl Criterion only requires us to "check" countably many matrix element functions, or equivalently where $\hat{G}$ is countable. It turns out that this is tantamount to a metrizability assumption.
	
	\begin{Lem}
		Let $G$ be a compact topological group. Then the following are equivalent.
		\begin{enumerate}
			\item $C(G)$ is separable as a vector space with the uniform norm.
			\item $L^2(G)$ is separable as a Hilbert space.
			\item The family $\hat{G}$ is countable.
			\item $G$ is metrizable.
		\end{enumerate}
	\end{Lem}
	
	\begin{proof}
		(1)$\Rightarrow$(2) If $C(G)$ admits a countable set with dense span in $C(G)$, then that same set has dense span in $L^2(G)$. Thus $L^2(G)$ is separable.
		
		(2)$\Rightarrow$(1) If $L^2(G)$ is separable, then every orthonormal basis of $L^2(G)$ is countable, including the family of matrix terms. But these matrix terms have dense span in $C(G)$, so $C(G)$ is separable.
		
		(2)$\iff$(3) If there are only countably many unitary equivalence classes of irreducible unitary representations of $G$ (which are all necessarily finite-dimensional), then $L^2(G)$ is separable by the Peter-Weyl Theorem. Conversely, if $\hat{G}$ is uncountable, then Peter-Weyl tells us that $L^2(G)$ admits an uncountable orthonormal basis, meaning that $L^2(G)$ is not separable.
		
		(4)$\Rightarrow$(1) If $G = (G, \rho)$ is compact metrizable, then $G$ is separable, admitting a countable dense subset $\{x_j\}_{j \in I}$ Let $f_j = \rho(\cdot, x_j)$. Then $\{ f_j \}_{j \in J}$ separates points, since if $f_{j}(x) = f_j(y)$ for all $j \in J$, then $\rho(x, x_j) = \rho(y, x_j)$ for all $j \in I$. For each $n \in \mathbb{N}$ exists $j_n \in I$ such that $\rho(x, x_{j_n}) \leq \frac{1}{2n}$, since $\{x_j\}_{j \in J}$ is dense in $G$. Thus $\rho(x, y) \leq \rho(x, x_j) + \rho(x_j, y) \leq \frac{1}{n}$. Thus $x = y$. By Stone-Weierstrass, this implies that
		$$\overline{\operatorname{span}}\left\{ \prod_{n = 1}^N f_{j_n} : j_1, \ldots, j_N \in J, N \in \mathbb{N} \cup \{0\} \right\} = C(G),$$
		where the empty product is the constant function $1$.
		
		(2)$\Rightarrow$(4) Let $\lambda : G \to \mathbb{U} \left( L^2(G) \right)$ be the left regular representation
		$$\lambda(x) : (t \mapsto f(t)) \mapsto \left(t \mapsto f\left( x^{-1} t \right) \right) .$$
		Then $\lambda$ is a faithful representation of $G$ on $\mathcal{H}$, meaning $\lambda$ is an embedding of $G$ into $\mathbb{U} \left( L^2(G) \right)$. Therefore $\lambda(G) \cong G$ is a closed subgroup of $\mathbb{U}\left( L^2(G) \right)$. But $\mathbb{U}(\mathcal{H})$ is metrizable when $\mathcal{H}$ is separable, so $G$ is therefore metrizable.
	\end{proof}
	
	When $G$ is compact abelian, the family $\hat{G}$ is exactly the family $\operatorname{Hom}\left( G, \mathbb{S}^1 \right)$ of continuous group homomorphisms $G \to \mathbb{S}^1$, where $\mathbb{S}^1 = \{z \in \mathbb{C} : z \overline{z} = 1\}$. Then $\hat{G}$ has the structure of a locally compact abelian topological group under pointwise multiplication \cite[1.2.6(d)]{RudinFourier}.
	
	The following results describes a condition under which certain sequences will be almost surely uniformly distributed. We remark that the result is a direct generalization of Theorem 4.1 from Chapter 1 of \cite{KN}, which proves the result (albeit in different language) for the particular case where $G = \mathbb{R} / \mathbb{Z}$, and our method of proof is essentially the same, except expressed in the language of harmonic analysis.
	
	\begin{Thm}\label{Difference Property Uniformly Distributed}
		Let $G$ be a compact abelian metrizable group, and let $\left( \Phi_n \right)_{n \in \mathbb{N}}$ be a sequence of distinct continuous group endomorphisms of $G$ such that $\Phi_n - \Phi_m$ is a surjection of $G$ onto itself for all $n \neq m$. Then for almost all $x \in G$, the sequence $( \Phi_n x )_{n = 0}^\infty$ is uniformly distributed.
	\end{Thm}
	
	\begin{proof}
		Fix some nontrivial irreducible unitary representation $\gamma \in \hat{G}$. Set
		$$S_\gamma(k, x) = \frac{1}{k} \sum_{i = 0}^{k - 1} \gamma(\Phi_i x) .$$
		Then
		\begin{align*}
			|S_\gamma(k, x)|^2	& = \frac{1}{k^2} \sum_{i, j = 0}^{k - 1} \gamma(\Phi_i x) \overline{\gamma(\Phi_j x)} \\
			& = \frac{1}{k^2} \sum_{i , j = 0}^{k - 1} \gamma((\Phi_i - \Phi_j) x) \\
			\Rightarrow \int |S_\gamma (k, x)|^2 \mathrm{d} \mu(x)	& = \frac{1}{k^2} \sum_{i, j = 0}^{k - 1} \int \gamma((\Phi_i - \Phi_j) x) \mathrm{d} \mu(x) \\
			& = \frac{1}{k} .
		\end{align*}
		This cancellation is possible because if $i \neq j$, then $\Phi_i - \Phi_j$ is surjective, meaning that $\gamma \circ (\Phi_i - \Phi_j)$ is a nontrivial character on $G$. Therefore $\int \gamma((\Phi_i - \Phi_j) x) \mathrm{d} \mu(x) = 0$ for $i \neq j$, meaning only the terms of $i = j$ contribute to the sum.
		
		This tells us that $\sum_{K = 1}^\infty \int \left|S_\gamma\left(K^2, x\right)\right|^2 \mathrm{d} \mu(x) = \sum_{K = 1}^\infty K^{-2} < \infty$, so by Fatou's Lemma we know that $\int \sum_{K = 1}^\infty \left| S_\gamma \left( K^2, x \right) \right|^2 \mathrm{d}  \mu(x) < \infty$. In particular, this tells us that $ \sum_{K = 1}^\infty \left| S_\gamma \left( K^2, x \right) \right|^2 < \infty$ for almost all $x \in G$, and so for almost all $x \in G$, we have $S_\gamma \left( K^2, x \right) \stackrel{K \to \infty}{\to} 0$.
		
		Let $x \in G$ such that $S_\gamma \left( K^2, x \right) \stackrel{K \to \infty}{\to} 0$. We want to show  that $S_\gamma \left( k, x \right) \stackrel{k \to \infty}{\to} 0$. For any $k \in \mathbb{N}$, we have
		$$\lfloor \sqrt{k} \rfloor^2 \leq k \leq \left( \lfloor \sqrt{k} \rfloor + 1 \right)^2 .$$
		So
		\begin{align*}
			\left| S_\gamma (k, x) \right|	& = \left| \frac{1}{k} \sum_{i = 0}^{k - 1} \gamma(\Phi_i x) \right| \\
			& \leq \left| \frac{1}{k} \sum_{i = 0}^{\lfloor \sqrt{k} \rfloor^2 - 1} \gamma(\Phi_i x) \right| + \left| \frac{1}{k} \sum_{j = \lfloor \sqrt{k} \rfloor^2}^{k - 1} \gamma(\Phi_j x) \right| \\
			& \leq \left| \frac{1}{\lfloor \sqrt{k} \rfloor^2} \sum_{i = 0}^{\lfloor \sqrt{k} \rfloor^2 - 1} \gamma(\Phi_i x) \right| + \frac{2 \lfloor \sqrt{k} \rfloor}{k} \\
			& = \left| S_\gamma \left( \lfloor \sqrt{k} \rfloor^2 , x \right) \right| + \frac{2 \lfloor \sqrt{k} \rfloor}{k} \\
			& \leq \left| S_\gamma \left( \lfloor \sqrt{k} \rfloor^2 , x \right) \right| + \frac{2}{\sqrt{k}} \\
			& \stackrel{k \to \infty}{\to} 0 .
		\end{align*}
		
		Let $E_\gamma = \left\{ x \in G : S_\gamma (k, x) \stackrel{k \to \infty}{\to} 0 \right\}$, and let $E = \bigcap_{\gamma \in \hat{G} \setminus \{1\}} E_\gamma.$ Since $\hat{G}$ is countable, we know $\mu(E) = 1$, proving the theorem.
	\end{proof}

	\section{The Difference Property}\label{Difference Property}
	
	We will be interested especially in the situation where the sequence $\left( \Phi_n \right)_{n \in \mathbb{N}}$ is generated by a sequence $(T_n)_{n \in \mathbb{N}}$ of continuous group endomorphisms. Motivated by Theorem \ref{Difference Property Uniformly Distributed}, we introduce the following definition.
	
	\begin{Def}
		Let $G$ be a compact abelian metrizable group, and let $(T_n)_{n = 1}^\infty$ be a sequence of continuous group endomorphisms of $G$. Set $\Phi_n = T_n T_{n - 1} \cdots T_1$ for $n \in \mathbb{N}$. We say that the sequence $(T_n)_{n = 1}^\infty$ has \emph{the Difference Property} if $\Phi_n - \Phi_m$ is surjective for all $n, m \in \mathbb{N}, n \neq m$.
	\end{Def}
	
	It is not obvious that the Difference Property places any particular restrictions on the individual $\Phi_n$ themselves. We have a special interest in when the maps $\{T_n\}_{n = 1}^\infty$ are surjective -or perhaps more interestingly, when they are \emph{not} surjective- because a continuous group endomorphism on a compact group is measure-preserving if and only if it is surjective.
	
	\begin{Prop}
		Let $T : G \to G$ be a continuous group endomorphism of a compact group $G$. Then $T$ is surjective if and only if $T$ is measure-preserving.
	\end{Prop}
	
	\begin{proof}
		$(\Rightarrow)$ Suppose $T$ is surjective. Define a Borel measure $\nu$ on $G$ by $\nu(E) = \mu \left( T^{-1} E \right)$. Let $x \in G$, and choose $y \in G$ such that $x = T y$ Then the measure $\nu$ satisfies
		\begin{align*}
			\nu \left( x E \right)	& = \mu \left( T^{-1} \left( x E \right) \right) \\
			& = \mu \left( T^{-1} \left( (Ty) E \right) \right) \\
			& = \mu \left( y T^{-1} E \right) \\
			& = \mu \left( T^{-1} E \right) \\
			& = \nu(E),
		\end{align*}
		meaning that $\nu$ is left-invariant. It also satisfies $\nu(G) = \mu \left( T^{-1} G \right) = \mu(G) = 1$, so $\nu$ is a probability measure. Therefore $\nu$ is a Haar probability measure on $G$, but by the uniqueness of the Haar measure, this implies that $\nu = \mu$.
		
		$(\Leftarrow)$ If $T$ is \emph{not} surjective, then there exists $x \in G \setminus TG$. Since $G$ is compact and Hausdorff, the map $T$ is necessarily closed, so $TG$ is closed in $G$, and a fortiori is measurable. If $\mu(TG) \neq 1$, then we know that $\mu \left( T^{-1} (TG) \right) = \mu(G) \neq \mu(TG)$, so consider the case where $\mu(TG) = 1$. We claim that $T^{-1} (x TG) = \emptyset$, since if $Ty = x Tz$ for some $y, z \in G$, then
		\begin{align*}
			T \left( y z^{-1} \right)	& = (Ty) \left( Tz^{-1} \right) \\
			& = x Tz (Tz)^{-1} \\
			& = x,
		\end{align*}
		a contradiction. Thus $\mu \left( T^{-1} \left( x TG \right) \right) = 0 \neq 1 = \mu \left( TG \right) = \mu(x TG)$.
	\end{proof}
	
	We note that our argument for the forward direction can be found in \cite[\S1.2]{Walters}. We include it here for the sake of a self-contained treatment.
	
	The possibility that in the non-autonomous case, rich results like Theorem \ref{Difference Property Uniformly Distributed} could be achieved where a nontrivial number of the $\{T_n\}_{n = 1}^\infty$ are not even measure-preserving intrigues us. The following results show that under certain conditions, the Difference Property imposes surjectivity on the $T_n$ individually.
	
	\begin{Prop}
		Let $G$ be a compact abelian metrizable group, and let $(T_n)_{n = 1}^\infty$ be a sequence of continuous group endomorphisms on $G$ with the Difference Property such that $T_n T_m = T_m T_n$ for all $m, n \in \mathbb{N}$. Then each $T_n$ is surjective.
	\end{Prop}
	
	\begin{proof}
		For each $n \in \mathbb{N}$, we have that
		\begin{align*}
			\Phi_{n + 1} - \Phi_n	& = T_{n + 1} \Phi_n - \Phi_n \\
			& = (T_{n + 1} - 1) \Phi_n \\
			& = \Phi_n (T_{n + 1} - 1)
		\end{align*}
		is surjective, meaning that $\Phi_n$ is surjective. Since $\Phi_n = T_n \Phi_{n - 1}$ (where we let $\Phi_0 = \operatorname{id}_G$), we can conclude that $T_n$ is surjective.
	\end{proof}
	
	This result applies in the autonomous setting, where $T_n = T$ for all $n \in \mathbb{N}$, but not in general. Even in the case of $\mathbb{R}^2 / \mathbb{Z}^2$, endomorphisms need not commute. The next example addresses the situation of finite-dimensional tori.
	
	\begin{Fact}\label{Characterizing compact abelian Lie groups}
		If $G$ is a compact connected abelian Lie group, then $G = \mathbb{R}^h / \mathbb{Z}^h$ for some $h \in \mathbb{N}$. In general, if $G$ is a compact abelian Lie group, then $G = G_0 \oplus B$, where $G_0$ is the identity component of $G$, and $B \cong A = G / G_0$.
	\end{Fact}
	
	\begin{proof}
		For a characterization of the compact connected Lie groups, see \cite[\S4.2]{Procesi}. We now argue that this implies that any compact abelian Lie group can be expressed in the way described above.
		
		We want to prove that $G$ is isomorphic to a direct sum of $G_0$ and $A = G / G_0$. The embedding $\iota_1 : G_0 \hookrightarrow G$ is just the canonical embedding, so it remains to find an embedding $\iota_2 : A \hookrightarrow G$.
		
		Since $A$ is finite abelian, there exist $a_1, \ldots, a_m \in A$, as well as $\ell_1, \ldots, \ell_m \geq 2$ such that
		$$A = \left< a_1 \right>_{\ell 1} \oplus \cdots \oplus \left< a_m \right>_{\ell_m} .$$ Let $\pi : G \to G / G_0$ be the canonical projection. For each $j \in \{1, \ldots, m\}$, choose $y_j \in G$ such that $\pi(y_j) = a_j$. Then $\ell_j y_j \in \ker \pi = G_0$. Since $G_0$ is divisible, there exists $y_j' \in G_0$ such that $\ell_j y_j' = \ell_j y_j$. Set $b_j = y_j - y_j'$, so that $\ell_j b_j = 0$. However, because $\pi(k b_j) = k \pi(a_j)$ for $k = 1, \ldots, \ell_j - 1$, we know that $b_j$ has order $\ell_j$ in $G$.
		
		Let $B = \oplus_{j = 1}^m \left< b_j \right>_{\ell_1}$ denote the subgroup of $G$ generated by $\{ b_1, \ldots, b_m \}$. Then $\pi : B \to A$ is an isomorphism. Let $\iota_2 = \pi^{-1} : A \to B \leq G$.
		
		We claim now that $\iota_1 : G_0 \hookrightarrow G, \iota_2 : A \hookrightarrow G$ generate $G$ as a direct sum of $G_0$ and $A$. First, we observe that $G_0 \cap B = \{0\}$, since if $x \in G_0 \cap B$, then $\pi(x) = 0$, because $G_0 = \ker \pi$, but also $x = \pi^{-1} (0) = 0$, since $\pi|_{B} : B \to A$ is an isomorphism. Therefore, the sum $G_0 + B$ is direct.
		
		Now, we have to show that $G = G_0 + B$. If $x \in G$, then choose $b \in B$ such that $\pi(b) = \pi(x)$. Then $x - b \in \ker \pi = G_0$. Thus $x = (x - b) + b$.
		
		Therefore $G = G_0 \oplus B \cong G_0 \oplus A$.
	\end{proof}
	
	\begin{Prop}
		Let $G$ be a compact connected abelian Lie group, and let $(T_n)_{n = 1}^\infty$ be a sequence of continuous group endomorphisms on $G$ with the Difference Property. Then each $T_n$ is surjective.
	\end{Prop}
	
	\begin{proof}
		Let $G = \mathbb{R}^h / \mathbb{Z}^h$. We can express $T_n$ as an $h \times h$ integer matrix. The maps $\Phi_n - \Phi_m$ are surjective if and only if $\det(\Phi_n - \Phi_m) \neq  0$. We can conclude that each $T_n$ is invertible because
		\begin{align*}
			\Phi_{n + 1} - \Phi_n	& = (T_{n + 1} - I) \Phi_n \\
			& = (T_{n + 1} - I) T_n T_{n - 1} \cdots T_1 \\
			\Rightarrow \det (\Phi_{n + 1} - \Phi_n)	& = \det(T_{n + 1} - I) \det (T_n) \det(T_{n - 1}) \cdots \det(T_1) \\
			& \neq 0 .
		\end{align*}
		This implies in particular that $\det (T_n) \neq 0$, meaning that $T_n$ is surjective.
	\end{proof}
	
	We can, however, provide \emph{negative} results on when a sequence with the Difference Property might exist.
	
	\begin{Lem}
		Let $G = G_0 \oplus A$ be a compact abelian group, where $G_0$ is a compact connected abelian group and $A$ is a finite abelian group. Let $\psi : G \to A$ be a continuous group homomorphism. Then $$\psi(G) = \psi(0 \oplus A) = \{\psi(0, a) : a \in A \} .$$
	\end{Lem}
	
	\begin{proof}
		Since $G_0$ is connected, we know that $\psi|_{G_0} : G_0 \to A$ is the zero map. Clearly $\psi(0 \oplus A) \subseteq \psi(G)$. Now suppose that $a \in \psi(G)$. Then there exist $x_0 \in G_0, a_0 \in A$ such that $\psi(x_0, a_0) = a$. But $\psi(0, a_0) = \psi(x_0, a_0) - \psi(x_0, 0) = \psi(x_0, a_0)$. Thus $a \in \psi(0 \oplus A)$.
	\end{proof}
	
	\begin{Prop}
		Let $G = G_0 \oplus A$, where $G_0$ is a compact connected abelian group, and $A$ a finite group with more than one element. Then there does not exist a sequence $(T_n)_{n = 1}^\infty$ on $G$ with the Difference Property. In particular, there does not exist a sequence with the Difference Property on any compact abelian Lie group that is not connected.
	\end{Prop}
	
	\begin{proof}
		Assume for contradiction that $(T_n)_{n = 1}^\infty$ has the Difference Property. Let $\Delta_{n, m} = \Phi_n - \Phi_m$. Let $\iota : A \hookrightarrow G_0 \oplus A$ be the canonical embedding. Then by the previous lemma, the maps $\Delta_{n, m} \circ \iota = \left( \Phi_n \circ \iota \right) - \left( \Phi_m \circ \iota \right)$ are surjective. But this is a contradiction, since $n \mapsto \Phi_n \circ \iota$ is a mapping from $\mathbb{N}$ to the finite set $A^A$, which cannot be injective. Therefore there exist $n, m \in \mathbb{N}, n \neq m$ such that $\Phi_n \circ \iota = \Phi_m \circ \iota$, meaning that $\Delta_{n, m}(G) = \Delta_{n, m} \circ \iota(A) = \{ 0 \} \neq A$, a contradiction.
		
		The special case of compact abelian Lie groups comes from Fact \ref{Characterizing compact abelian Lie groups}.
	\end{proof}

	\section{A random non-autonomous spatial-temporal differentiation problem}\label{Probabilistic results}
	
	\begin{Lem}\label{Distribution and STDs}
		Let $G = (G, \rho)$ be a compact metrizable group with metric $\rho$. Let $(T_n)_{n = 1}^\infty$ be a family of Lipschitz-continuous maps $T_n : G \to G$, where $T_n$ is $L_n$-Lipschitz. Let $\Phi_n = T_n T_{n - 1} \cdots T_1, \Phi_0 = \operatorname{id}_G$. Then there exists a sequence $(\eta_k)_{k = 1}^\infty$ of positive numbers $\eta_k > 0, \eta_k \stackrel{k \to \infty}{\to} 0$ such that if $(r_k)_{k = 1}^\infty$ is a sequence of positive numbers $0 < r_k \leq \eta_k$, and $B_k(x) = B(x, r_k)$, then for every $f \in C(G)$, and every sequence $(x_n)_{n = 1}^\infty \in G^\mathbb{N}$, we have
		$$\left| \alpha_{B_k(x_k)} \left( \frac{1}{k} \sum_{i = 0}^{k - 1} \Phi_i f \right) - \frac{1}{k} \sum_{i = 0}^{k - 1} \Phi_i f(x_k) \right| \stackrel{k \to \infty}{\to} 0 .$$
	\end{Lem}
	
	\begin{proof}
		Let $\tilde{L}_k = \max\{ 1, L_1, L_2, \ldots, L_{k - 1} \}$. Set
		$$\eta_k = \tilde{L}_k^{- (k - 1)} k^{-1} .$$ Let $(r_k)_{k = 1}^\infty$ be a sequence such that $0 < r_k \leq \eta_k$. Then each $T_1, \ldots, T_{k - 1}$ is $\tilde{L}_k$-Lipschitz, so if $y \in B_k(x)$, i.e. if $\rho(x, y) < r_k$, and if $i \in [0, k - 1]$, then
		\begin{align*}
			\rho \left( \Phi_i x, \Phi_i y \right)	& = \rho \left( T_i T_{i - 1} \cdots T_1 x, T_i T_{i - 1} \cdots T_1 y \right) \\
			& \leq L_i L_{i - 1} \cdots L_1 \rho(x, y) \\
			& \leq \left( \tilde{L}_k \right)^i \rho(x, y) \\
			& \leq \left( \tilde{L}_k \right)^{i - (k - 1)} k^{-1} \\
			& \leq k^{-1} .
		\end{align*}
		
		Now let $f \in C(G)$, and fix $\epsilon > 0$. By uniform continuity of $f$, there exists $K \in \mathbb{N}$ such that $\rho(z, w) \leq \frac{1}{K} \Rightarrow |f(z) - f(w)| \leq \epsilon$. Then if $k \geq K$, then
		\begin{align*}
			\left| \alpha_{B_k(x_k)} \left( \frac{1}{k} \sum_{i = 0}^{k - 1} \Phi_i f \right) - \frac{1}{k} \sum_{i = 0}^{k - 1} \Phi_i f(x_k) \right|	& = \frac{1}{k} \left| \sum_{i = 0}^{k - 1} \Phi_i f(x_k) - \frac{1}{\mu(B_k(x_k))} \int_{B_k(x_k)} \Phi_i f \mathrm{d} \mu \right| \\
			& \leq \frac{1}{k} \sum_{i = 0}^{k - 1} \left| \Phi_i f(x_k) - \frac{1}{\mu(B_k(x_k))} \int_{B_k(x_k)} \Phi_i f \mathrm{d} \mu \right| \\
			& = \frac{1}{k} \sum_{i = 0}^{k - 1} \left| \frac{1}{\mu(B_k(x_k))} \int_{B_k(x_k)} \left( \Phi_i f(x_k) - \Phi_i f \right) \mathrm{d} \mu \right| \\
			& \leq \frac{1}{k} \sum_{i = 0}^{k - 1} \frac{1}{\mu(B_k(x_k))} \int_{B_k(x_k)} \left| \left( \Phi_i f(x_k) - \Phi_i f \right)\right| \mathrm{d} \mu \\
			& \leq \frac{1}{k} \sum_{i = 0}^{k - 1} \frac{1}{\mu(B_k(x_k))} \int_{B_k(x_k)} \epsilon \mathrm{d} \mu \\
			& = \epsilon .
		\end{align*}
		Therefore
		$$\left| \alpha_{B_k(x_k)} \left( \frac{1}{k} \sum_{i = 0}^{k - 1} \Phi_i f \right) - \frac{1}{k} \sum_{i = 0}^{k - 1} \Phi_i f(x_k) \right| \stackrel{k \to \infty}{\to} 0 .$$
	\end{proof}
	
	Our assumption that the $\{T_n\}_{n = 1}^\infty$ are Lipschitz is not overly restrictive. In particular, if $G$ has the structure of a Lie group with a Riemannian metric, and the maps $T_n$ are $C^1$, then the maps $T_n$ are automatically Lipschitz in that Riemannian metric.
	
	\begin{Thm}\label{Main Theorem}
		Let $G = (G, \rho)$ be a compact abelian metrizable group with metric $\rho$. Let $(T_n)_{n = 1}^\infty$ be a family of Lipschitz-continuous group homomorphisms with the Difference Property, where $T_n$ is $L_n$-Lipschitz. Then there exists a Borel subset $E$ of full measure and a sequence $(\eta_k)_{k = 1}^\infty$ of positive numbers $\eta_k > 0, \eta_k \stackrel{k \to \infty}{\to} 0$ such that if $(r_k)_{k = 1}^\infty$ is a sequence of positive numbers $0 < r_k \leq \eta_k$, and $B_k(x) = B(x, r_k)$, then for all $x \in E$, we have
		\begin{align*}
			\alpha_{B_k(x)} \left( \frac{1}{k} \sum_{i = 0}^{k - 1} \Phi_i f \right)	& \stackrel{k \to \infty}{\to} \int f \mathrm{d} \mu	& (\forall f \in C(G)) .
		\end{align*}
	\end{Thm}
	
	\begin{proof}
		Let $E = \left\{x \in G : \left( \Phi_n x \right)_{n = 0}^\infty \textrm{ is uniformly distributed in $G$} \right\}$. By Theorem \ref{Difference Property Uniformly Distributed}, this set is of full measure. Now let $(\eta_k)_{k = 1}^\infty$ be as in Lemma \ref{Distribution and STDs}. Then if $0 < r_k \leq \eta_k$, and if $x \in E$, then for every $f \in C(G)$ we have
		\begin{align*}
			\left| \alpha_{B_k(x)} \left( \frac{1}{k} \sum_{i = 0}^{k - 1} \Phi_i f \right) - \int f \mathrm{d} \mu \right|	& \leq \left| \alpha_{B_k(x)} \left( \frac{1}{k} \sum_{i = 0}^{k - 1} \Phi_i f \right) - \frac{1}{k} \sum_{i = 0}^{k - 1} \Phi_i f(x) \right| + \left| \left( \frac{1}{k} \sum_{i = 0}^{k - 1} \Phi_i f (x) \right) - \int f \mathrm{d} \mu \right| \\
			& \stackrel{k \to \infty}{\to} 0 ,
		\end{align*}
		where the first term goes to $0$ by Lemma \ref{Distribution and STDs}, and the second term goes to $0$ by the fact that $\left( \Phi_{n} x \right)_{n = 0}^\infty$ is  uniformly distributed in $G$.
	\end{proof}

	\section{Further probabilistic results about uniformly distributed sequences}\label{Not concentric}

	We now consider the distribution properties of randomly chosen sequences $(x_n)_{n = 0}^\infty \in G^{\mathbb{N}_0} = \prod_{n = 0}^\infty G$ in $G$.
	
	\begin{Lem}\label{Borel algebra of products}
		Let $(X_n)_{n \in \mathbb{N}}$ be a sequence of separable metrizable topological spaces. Then
		$$\operatorname{Bo} \left( \prod_{n \in \mathbb{N}} X_n \right) = \bigotimes_{n \in \mathbb{N} } \operatorname{Bo} (X_n) .$$
	\end{Lem}
	
	\begin{proof}
		\cite[Lemma 1.2]{Kallenberg}
	\end{proof}
	
	\begin{Def}
		Let $X$ be a nonempty set. We call $\mathcal{A} \subseteq \mathcal{P}(X)$ a \emph{semi-algebra} if
		\begin{enumerate}[label=(\alph*)]
			\item $\emptyset \in \mathcal{A}$
			\item If $A, B \in \mathcal{A}$, then $A \cap B \in \mathcal{A}$.
			\item For every $A \in \mathcal{A}$, the set $X \setminus A$ can be written as a disjoint union of finitely many elements of $\mathcal{A}$.
		\end{enumerate}
	\end{Def}
	
	\begin{Lem}\label{Generating semi-algebra}
		Let $\left( X_n, \mathcal{B}_n \right)_{n \in \mathbb{N}}$ be a sequence of measurable spaces, and let
		$$\mathcal{A} = \left\{ B_1 \times \cdots \times B_N \times \prod_{n = N + 1}^\infty X_n : N \in \mathbb{N}, B_1 \in \mathcal{B}_1, \ldots, B_N \in \mathcal{B}_N \right\} .$$
		Then $\mathcal{A}$ is a semi-algebra that generates $\bigotimes_{n \in \mathbb{N} } \mathcal{B}_n$.
	\end{Lem}
	
	\begin{proof}
		For $n \in \mathbb{N}$, let $\pi_i : \prod_{n \in \mathbb{N}} X_n \twoheadrightarrow X_i$ be the map $\pi_i : (x_n)_{n \in \mathbb{N}} \mapsto x_i$. Then $\mathcal{A}$ can be written as
		$$\mathcal{A} = \left\{ \bigcap_{n = 1}^N \pi_n^{-1} (B_n) : N \in \mathbb{N}, B_1 \in \mathcal{B}_1, \ldots, B_N \in \mathcal{B}_N \right\} .$$
		Written this way, it is clear that $\emptyset \in \mathcal{A}$ and that $\mathcal{A}$ is closed under finite intersections. Finally, we will prove that the complement of every set in $\mathcal{A}$ can be expressed as the disjoint union of finitely many elements of $\mathcal{A}$. Let $A = \bigcap_{n = 1}^N \pi_n^{-1} (B_n)$, where $B_1 \in \mathcal{B}_1, \ldots, B_N \in \mathcal{B}_N$. Then
		\begin{align*}
			A^\complement	& = \left( \bigcap_{n = 1}^N \pi_n^{-1} (B_n) \right)^\complement \\
			& = \bigcup_{n = 1}^N \pi_n^{-1} \left( B_n^\complement \right) \\
			& = \bigsqcup_{ I \in \mathcal{P}(\{1, 2, \ldots, N\}) } \left[ \left( \bigcap_{i \in I} \pi_i^{-1}(B_i)^\complement \right) \cap \left( \bigcap_{j \in \{1, \ldots, N\} \setminus I } \left( \pi_j^{-1}(B)^\complement \right)^\complement \right) \right] \\
			& = \bigsqcup_{ I \in \mathcal{P}(\{1, 2, \ldots, N\}) } \left[ \left( \bigcap_{i \in I} \pi_i^{-1} \left(B_i^\complement \right) \right) \cap \left( \bigcap_{j \in \{1, \ldots, N\} \setminus I } \pi_i^{-1}(B_j) \right) \right] .
		\end{align*}
		Therefore, we have written $A^\complement$ as a disjoint union of elements of $\mathcal{A}$.
		
		Finally, to justify our claim that $\mathcal{A}$ generates $\bigotimes_{n \in \mathbb{N} } \mathcal{B}_n$ as a $\sigma$-algebra, we note that $\bigotimes_{n \in \mathbb{N} } \mathcal{B}_n$ is generated as a $\sigma$-algebra by $\left\{ \pi_n^{-1} (B_n) : n \in \mathbb{N}, B_n \in \mathcal{B}_n \right\}$, and
		$$\left\{ \pi_n^{-1} (B_n) : n \in \mathbb{N}, B_n \in \mathcal{B}_n \right\} \subseteq \mathcal{A} \subseteq \bigotimes_{n \in \mathbb{N}} \mathcal{B}_n .$$
	\end{proof}
	
	We go to the trouble of proving Lemma \ref{Generating semi-algebra} because when checking whether a map between probability spaces is measure-preserving, it suffices to see how the map behaves on a generating semi-algebra, as in the following result.
	
	\begin{Lem}\label{Semi-algebras and pmp}
		Let $(X_1, \mathcal{B}_1, \mu_1) , (X_2, \mathcal{B}_2, \mu_2)$ be probability spaces, and let $T : X_1 \to X_2$ be a map. Let $\mathcal{A}_2 \subseteq \mathcal{B}_2$ be a semi-algebra which generates the $\sigma$-algebra $\mathcal{B}_2$. Then $T$ is measurable and measure-preserving iff $T^{-1} A \in \mathcal{B}_1$ for all $A \in \mathcal{A}_2$, and $\mu_1 \left( T^{-1} A \right) = \mu_2(A)$.
	\end{Lem}
	
	\begin{proof}
		\cite[Theorem 1.2.2]{DajaniDirksin}
	\end{proof}
	
	With all of this out of the way, we are ready to demonstrate a characterization of the Haar measure of a countable product of compact metrizable groups.
	
	\begin{Thm}\label{Haar measure of countable product}
	Let $(H_n)_{n \in \mathbb{N}}$ be a sequence of compact metrizable groups, and let $\nu_n$ be the left-invariant (resp. right-invariant) Haar probability measure of $H_n$. Then $\mu = \prod_{n \in \mathbb{N} } \nu_n$ is the left-invariant (resp. right-invariant) Haar probability measure on $G = \prod_{n \in \mathbb{N}} H_n$.
	\end{Thm}
	
	\begin{proof}
		We will demonstrate the claim for left-invariant Haar measures, since the proof for the right-invariant claim is essentially identical.
		
		We will show that $\mu$ is a $G$-invariant Borel probability measure on $G$, and then conclude from the uniqueness of the Haar measure that $\mu$ must be \emph{the} Haar probability measure on $G$. By Lemma \ref{Borel algebra of products}, we know that $\operatorname{Bo}(G) = \bigotimes_{ n \in \mathbb{N} } \operatorname{Bo}(H_n)$. Set
		$$\mathcal{A} = \left\{ B_1 \times \cdots \times B_N \times \prod_{n = N + 1}^\infty H_n : N \in \mathbb{N}, B_1 \in \operatorname{Bo}(H_1) , \ldots , B_N \in \operatorname{Bo}(H_N) \right\} .$$
		Then by Lemma \ref{Generating semi-algebra}, this $\mathcal{A}$ is a generating semi-algebra for $\bigotimes_{ n \in \mathbb{N} } \operatorname{Bo}(H_n)$.
		
		Now, fix $g = (h_n)_{n \in \mathbb{N} } \in G$. We want to prove that left multiplication on $G$ by $g$ is a $\mu$-preserving transformation. By Lemma \ref{Semi-algebras and pmp}, it will suffice to prove that $\mu(g A) = \mu(A)$ for all $A \in \mathcal{A}$. So fix sets $B_1 \in \operatorname{Bo}(H_1), \ldots, B_N \in \operatorname{Bo}(H_N), N \in \mathbb{N}$. Then
		\begin{align*}
			\mu \left( g \left( B_1 \times \cdots \times B_N \times \prod_{n = N + 1}^\infty H_n \right) \right)	& = \mu \left( h_1 B_1 \times \cdots \times h_N B_N \times \prod_{n = N + 1}^\infty h_n H_n \right) \\
			& = \mu \left( h_1 B_1 \times \cdots \times h_N B_N \times \prod_{n = N + 1}^\infty H_n \right) \\
			& = \nu_1(h_1 B_1) \cdots \nu_N(h_N B_N) \\
			& = \nu_1(B_1) \cdots \nu_N (B_N) \\
			& = \mu \left( B_1 \times \cdots \times B_N \times \prod_{n = N + 1}^\infty H_n \right) .
		\end{align*}
		
		We have thus established that $\mu$ is $G$-invariant on $\mathcal{A}$, and so we can infer that $\mu$ is $G$-invariant for all of $\operatorname{Bo}(G)$.
	\end{proof}
	
	For the remainder of this section, let $G$ be a compact abelian metrizable group. Let $S : G^{\mathbb{N}_0} \twoheadrightarrow G^{\mathbb{N}_0}$ be the left shift
	$$S (g_n)_{n = 0}^\infty = (g_{n + 1})_{n = 0}^\infty .$$
	Then $S$ is a continuous surjective group endomorphism of $G^{\mathbb{N}_0}$, and given a continuous group homomorphism $T : G \to G$, let $\widehat{T} : G^{\mathbb{N}_0} \to G^{\mathbb{N}_0}$ be the map
	$$\widehat{T} : (g_n)_{n = 0}^\infty \mapsto (T g_n)_{n = 0}^\infty .$$
	We can observe that $S$ and $\widehat{T}$ commute, since
	\begin{align*}
		S \widehat{T} (g_n)_{n = 0}^\infty	& = S \left( T g_n \right)_{n = 0}^\infty \\
		& = \left( T g_{n + 1} \right)_{n = 0}^\infty \\
		& = \widehat{T} (g_{n + 1})_{n = 0}^\infty \\
		& = \widehat{T} S (g_n)_{n = 0}^\infty .
	\end{align*}
	
	\begin{Lem}\label{Lemma towards DP}
		Let $T : G \twoheadrightarrow G$ be a continuous surjective group endomorphism of $G$, and fix $\ell \in \mathbb{N}$. Then $\operatorname{id}_{G^{\mathbb{N}_0}} - S^\ell \widehat{T}$ is also surjective.
	\end{Lem}
	
	\begin{proof}
		Fix $g = (g_n)_{n = 0}^\infty \in G^{\mathbb{N}_0}$, and construct a sequence $g' = \left( g_n' \right)_{n = 0}^\infty \in G^{ \mathbb{N}_0 }$ recursively as follows. First, set $g_n' = g_n$ for $n = 0, \ldots, \ell - 1$. Then for $N > \ell - 1$, assuming that $g_0', g_1', g_2', \ldots, g_N' \in G$ have been chosen such that 
		\begin{align*}
			g_n ' - T g_{n + \ell} ' & = g_n	& (\textrm{for $n = 0, 1 , 2 , \ldots, N - \ell$}) ,
		\end{align*}
		choose $g_{N + 1} ' \in G$ such that
		$$T g_{N + 1} ' = g_{N + 1 - \ell} - g_{N + 1 - \ell} ' ,$$
		which exists because $T$ is surjective. Then
		$$g_{N + 1 - \ell} ' - T g_{N + 1} ' = g_{N + 1 - \ell} .$$
		Continuing this process gives us a sequence $g' = \left( g_n' \right)_{n = 0}^\infty$ such that
		$$\left( \operatorname{id}_{G^{\mathbb{N}_0} } - S^\ell \widehat{T} \right) g' = g .$$
	\end{proof}
	
	\begin{Lem}\label{Shifts have DP}
		Let $(T_n)_{n = 1}^\infty$ be a sequence of surjective group homomorphisms $T_n : G \twoheadrightarrow G$, and let $(\ell_n)_{n = 1}^\infty$ be a sequence of natural numbers. Then the sequence $\left( S^{\ell_n} \widehat{T}_n \right)_{n = 1}^\infty$ has the Difference Property on $G^{\mathbb{N}_0}$.
	\end{Lem}
	
	\begin{proof}
		Set $\Lambda_n = \ell_1 + \cdots + \ell_n$. Let $m, n \in \mathbb{N} , m < n$. Then
		\begin{align*}
			\left( \left( S^{\ell_m} \widehat{T}_m \right) \cdots \left( S^{\ell_1} \widehat{T}_1 \right) \right) - \left( \left( S^{\ell_n} \widehat{T}_n \right) \cdots \left( S^{\ell_1} \widehat{T}_1 \right) \right)	& = \left( S^{\Lambda_m} \widehat{T}_m \cdots \widehat{T}_1 \right) - \left( S^{\Lambda_n} \widehat{T}_n \cdots \widehat{T}_1 \right) \\
			& = \left( \widehat{T}_m \cdots \widehat{T}_1 S^{\Lambda_m} \right) - \left( \widehat{T}_n \cdots \widehat{T}_1 S^{\Lambda_n} \right) \\
			& = \left( \operatorname{id}_{G^{\mathbb{N}_0} } - \widehat{T}_n \widehat{T}_{n - 1} \cdots \widehat{T}_{m + 1} S^{\Lambda_n - \Lambda_m} \right) S^{\Lambda_m} \widehat{\Phi_m} \\
			& = \left( \operatorname{id}_{G^{\mathbb{N}_0} } - \widehat{\tau(n, m)} S^{\Lambda_n - \Lambda_m} \right) S^{\Lambda_m} \widehat{\Phi_m} \\
			& = \left( \operatorname{id}_{G^{\mathbb{N}_0} } - S^{\Lambda_n - \Lambda_m} \widehat{\tau(n, m)} \right) S^{\Lambda_m} \widehat{\Phi_m}
		\end{align*}
		where $\tau(n, m) = T_n T_{n - 1} \cdots T_{m + 1}$. By Lemma \ref{Lemma towards DP}, it follows that $\operatorname{id}_{G^{\mathbb{N}_0} } - \widehat{\tau(n, m)} S^{\Lambda_n - \Lambda_m}$ is surjective. Therefore $\left( \left( S^{\ell_m} \widehat{T}_m \right) \cdots \left( S^{\ell_1} \widehat{T}_1 \right) \right) - \left( \left( S^{\ell_n} \widehat{T}_n \right) \cdots \left( S^{\ell_1} \widehat{T}_1 \right) \right) = \left( \operatorname{id}_{G^{\mathbb{N}_0} } - S^{\Lambda_n - \Lambda_m} \widehat{\tau(n, m)} \right) S^{\Lambda_m}$ is surjective, since it is a composition of surjections.
	\end{proof}
	
	\begin{Cor}\label{Sequences are almost surely uniformly distributed}
		Let $(T_n)_{n = 1}^\infty$ be a sequence of surjective group homomorphisms $T_n : G \twoheadrightarrow G$, and let $(\ell_n)_{n = 1}^\infty$ be a sequence of natural numbers. Set $\Lambda_n = \ell_1 + \cdots + \ell_n$. Then for almost every $g \in G$, the sequence $\left( S^{\Lambda_n} \widehat{T}_n \cdots \widehat{T}_1 g \right)_{n = 0}^\infty$ is uniformly distributed in $G^{\mathbb{N}_0}$.
	\end{Cor}

	\begin{proof}
	Apply Lemma \ref{Shifts have DP} and Theorem \ref{Difference Property Uniformly Distributed}.
	\end{proof}
	
	\begin{Thm}\label{Shifts and uniform distribution}
		Let $(T_n)_{n = 1}^\infty$ be a sequence of surjective group homomorphisms $T_n : G \twoheadrightarrow G$, and let $(\ell_n)_{n = 1}^\infty$ be a sequence of natural numbers. Set $\Lambda_n = \ell_1 + \cdots + \ell_n$. Then for almost every sequence $(g_n)_{n = 0}^\infty \in G^{\mathbb{N}_0}$, the sequence $\left( T_n \cdots T_1 g_{\Lambda_n} \right)_{n = 0}^\infty$ is uniformly distributed in $G$.
	\end{Thm}
	
	\begin{proof}
	Let $\pi : G^{\mathbb{N}_0} \twoheadrightarrow G$ be the projection onto the first term $\pi : (g_n)_{n = 0}^\infty \mapsto g_0$. By Corollary \ref{Sequences are almost surely uniformly distributed}, for almost every $(g_n)_{n = 0}^\infty \in G^{\mathbb{N}_0}$, the sequence $\left( S^{\Lambda_n} \widehat{T}_n \cdots \widehat{T}_1 g \right)_{n = 0}^\infty$ is uniformly distributed in $G^{\mathbb{N}_0}$.
	
	Now let $(g_n)_{n = 0}^\infty \in G^{\mathbb{N}_0}$ be such that the sequence $\left( S^{\Lambda_n} \widehat{T}_n \cdots \widehat{T}_1 g \right)_{n = 0}^\infty$ is uniformly distributed in $G^{\mathbb{N}_0}$. Then by Lemma \ref{Group homomorphisms preserve uniform distribution}, the sequence $\left( \pi \left( S^{\Lambda_n} \widehat{T}_n \cdots \widehat{T}_1 g \right) \right)_{n = 0}^\infty$ is uniformly distributed in $G$. But
	\begin{align*}
	\pi \left( S^{\Lambda_n} \widehat{T}_n \cdots \widehat{T}_1 g \right)	& = T_n \cdots T_1 g_{\Lambda_n} .
	\end{align*}
	\end{proof}

	\section{Topologically generic behaviors of random spatial-temporal differentiation problems}\label{Topological results}
	
	We can interpret Theorem \ref{Difference Property Uniformly Distributed} as saying that if $(T_n)_{n = 1}^\infty$ is a sequence of continuous group endomorphisms of $G$ with the Difference Property, and $\Phi_n = T_n T_{n - 1} \cdots T_1$, then the property of $x \in G$ that $\left( \Phi_n x \right)_{n = 0}^\infty$ is uniformly distributed is "probabilistically generic", in the sense that the set of such $x$ has full measure. In light of Theorem \ref{Main Theorem}, we can infer that if $B_k(x)$ is a sequence of balls around $x$ with radii going to $0$ sufficiently fast, then it is probabilistically generic that
	\begin{align*}
		\alpha_{B_k(x)} \left( \frac{1}{k} \sum_{i = 0}^{k - 1} \Phi_i f \right)	& \stackrel{k \to \infty}{\to} \int f \mathrm{d} \mu	& (\forall f \in C(G)) .
	\end{align*}
	In other words, we can see Theorem \ref{Main Theorem} as a statement about a probabilistically generic spatial-temporal differentiation of a certain kind.
	
	However, if we try to look at \emph{topologically} generic behaviors, the story changes. Instead, in a sense that we will make precise momentarily, the topologically generic behavior is that the sequence $\left( \alpha_{B_k(x)} \left( \frac{1}{k} \sum_{i = 0}^{k - 1} \Phi_i f \right) \right)_{k = 1}^\infty$ is divergent for some $f \in C(G)$.
	
	\begin{Def}
	Let $X$ be a compact metrizable space, and $S \subseteq X$ a subset. We say that $A \subseteq X$ is \emph{nowhere dense} if for every nonempty open $\mathcal{O} \subseteq X$, there exists a nonempty open subset $W \subseteq \mathcal{O}$ such that $W \cap A = \emptyset$. A subset $A \subseteq X$ is called \emph{meager} if there exists a sequence $(A_n)_{n \in \mathbb{N}}$ of nowhere dense subsets of $X$ such that $A \subseteq \bigcup_{n \in \mathbb{N}} A_n$. We call a subset $B \subseteq X$ \emph{comeager} if $X \setminus B$ is meager. A comeager set is sometimes called \emph{Baire generic}.
	\end{Def}

	Our goal here is to show that the behavior described in Theorem \ref{Difference Property Uniformly Distributed} is -from this topological perspective- exceptional in the sense that the set of such $x \in G$ is meager.
	
	\begin{Thm}\label{Mance Theorem}
	Let $G = (G, \rho)$ be a compact abelian metrizable group with infinitely many elements and metric $\rho$. Let $(T_n)_{n = 1}^\infty$ be a sequence of continuous, surjective group endomorphisms $T_n : G \to G$. Let $\Phi_n = T_n T_{n - 1} \cdots T_1, \Phi_0 = \operatorname{id}_G$. Suppose that $\bigcup_{m = 1}^\infty \ker \Phi_m$ is dense in $G$. Then the set of $x \in G$ such that $\left( \Phi_n x \right)_{n = 0}^\infty$ is uniformly distributed is meager.
	\end{Thm}

Before we can prove Theorem \ref{Mance Theorem}, we need to prove a few technical lemmas.

	\begin{Lem}\label{Small balls}
	Let $G = (G, \rho)$ be a compact metrizable group with infinitely many elements and $\epsilon > 0$. Then there exists $\delta > 0$ such that $\mu(B(0, \delta)) < \epsilon$.
	\end{Lem}

\begin{proof}
	Consider the sequence $\left(B(0, 1/n)\right)_{n = 1}^\infty$. Then $B(0, 1) \supseteq B(0, 1/2) \supseteq B(0, 1/3) \supseteq \cdots$, and $\bigcap_{n = 1}^\infty B(0, 1/n) = \{0\}$. But $\mu(\{0\}) = 0$, so by continuity of measure, it follows that $\lim_{n \to \infty} \mu(B(0, 1/n)) = 0$. Thus in particular there exists $N \in \mathbb{N}$ such that $\mu(B(0, 1/N)) < \epsilon$. Let $\delta = 1/N$.
\end{proof}

\begin{Lem}\label{Concentrated at zero}
Let $G = (G, \rho)$ be a compact abelian metrizable group with infinitely many elements and metric $\rho$. Let $(T_n)_{n = 1}^\infty$ be a sequence of continuous, surjective group endomorphisms $T_n : G \to G$. Let $\Phi_n = T_n T_{n - 1} \cdots T_1, \Phi_0 = \operatorname{id}_G$. Suppose that $\bigcup_{m = 1}^\infty \ker \Phi_m$ is dense in $G$. Let $\epsilon, \delta > 0, N_1 \in \mathbb{N}$, and let $\mathcal{O} \subseteq G$ be a nonempty open set. Then there exists a nonempty open set $W \subseteq \mathcal{O}$ and $L \geq N_1$ such that if $x \in W$, then
$$\frac{\# \left\{ j \leq L - 1 : \Phi_j x \in B(0, \delta / 2) \right\}}{L} \geq 1 - \epsilon .$$
\end{Lem}

\begin{proof}
Choose $m \in \mathbb{N} , a \in \ker \Phi_m$ such that $a \in \mathcal{O}$. Choose $N_2 \in \mathbb{N}$ such that
$$\frac{m}{m + N_2} < \epsilon .$$
Set $L_0 = \max \{N_1, N_2\}$. Let $U \subseteq G$ be the open neighborhood of $0$ given by
$$U = \bigcap_{\ell = 0}^{L_0 - 1} \tau(m + \ell, m)^{-1} B(0, \delta / 2) .$$
Finally set
$$W = \left( \Phi_m^{-1} U \right) \cap \mathcal{O} .$$
Then $a \in W$, and $\Phi_{m + \ell} W \subseteq B(0, \delta / 2)$ for all $\ell \in \{0, 1, 2, \ldots, L_0 - 1 \}$. Let $L = L_0 + m$. Then
$$\frac{ \# \left\{ j \leq L - 1 : \Phi_j x \in B(0, \delta / 2) \right\} }{L_0 + m} \geq \frac{L_0}{L_0 + m}	= 1 - \frac{m}{m + L_0} \geq 1 - \frac{m}{m + N_2} \geq 1 - \epsilon .$$
\end{proof}

\begin{proof}[Proof of Theorem \ref{Mance Theorem}]
Let $f : G \to [0, 1]$ be the continuous function
$$f(x) = \begin{cases}
	1	& \rho(x, 0) \leq \frac{\delta}{2} , \\
	2 - \frac{2}{\delta} \rho(x, 0)	& \frac{\delta}{2} \leq \rho(x, 0) \leq \delta , \\
	0	& \rho(x, 0) \geq \delta ,
\end{cases}$$
where $\delta > 0$ is chosen such that $\mu(B(0, \delta)) < \frac{1}{2}$, which exists by Lemma \ref{Small balls}. Then $\chi_{B(0, \delta / 2)} \leq f \leq \chi_{B(0, \delta)}$, so
$$\mu(B(0, \delta / 2)) \leq \int f \mathrm{d} \mu \leq \mu(B(0, \delta)) < \frac{1}{2} . $$

For $K \in \mathbb{N}$, let $A_K$ be the set
$$A_K = \left\{ x \in G : \textrm{$\frac{1}{k} \sum_{i = 0}^{k - 1} f \left( \Phi_i x \right) < \frac{2}{3}$ for all $k \geq K$} \right\} .$$
We claim the set $A_K$ is nowhere dense. But the set of $x \in G$ such that $\left( \Phi_n x \right)_{n = 0}^\infty$ is uniformly distributed is contained in $\bigcup_{K = 1}^\infty A_K$, so if we can show that $A_K$ is nowhere dense for all $K \in \mathbb{N}$, then the theorem will be proven. Now fix $K \in \mathbb{N}$.

Let $\mathcal{O} \subseteq G$ be a nonempty open subset of $G$. By Lemma \ref{Concentrated at zero}, there exists a nonempty open subset $W \subseteq \mathcal{O}$ such that
$$\frac{\# \left\{ j \leq L - 1 : \Phi_j x \in B(0, \delta / 2) \right\}}{L} \geq \frac{2}{3} $$
for $x \in W$, where $L \geq K$. So
$$\frac{1}{k} \sum_{i = 0}^{k - 1} f \left( \Phi_i x \right) \geq \frac{1}{k} \sum_{i = 0}^{k - 1} \chi_{B(0, \delta / 2)} \left( \Phi_i x \right) \geq \frac{2}{3} .$$
Therefore $W \subseteq \mathcal{O} \setminus A_K$.
\end{proof}

Our proof of Theorem \ref{Mance Theorem} is based off of \cite[Theorem 8.3.1]{ManceDissertation}. That theorem can be interpreted as a special case of Theorem \ref{Mance Theorem} in the case where $G = \mathbb{R} / \mathbb{Z}$, though it is stated there in the language of normality with respect to a Cantor series.

However, this result can be strengthened under some mild additional assumptions. Theorem \ref{Mance Theorem} states that the family of $x \in G$ for which $\frac{1}{k} \sum_{i = 0}^{k - 1} f \left( \Phi_i x \right) \stackrel{k \to \infty}{\to} \int f \mathrm{d} \mu$ for all $f \in C(G)$ is meager. However, it can by shown that under some additional assumptions, there exists $f \in C(G)$ such that the family of $x \in G$ for which $\lim_{k \to \infty} \frac{1}{k} \sum_{i = 0}^{k - 1} f \left( \Phi_i x \right)$ exists is meager.

\begin{Lem}\label{Changes on a subset of density $0$}
Let $f \in C(G)$, and let $(x_n)_{n = 0}^\infty$ be a sequence in $G$ such that
$$\frac{1}{k} \sum_{i = 0}^{k - 1} f (x_i) \stackrel{k \to \infty}{\to} \lambda$$
for some $\lambda \in \mathbb{C}$. Let $I \subseteq \mathbb{N}_0$ be a subset of density $0$, i.e. such that
$$\frac{1}{k} \sum_{i = 0}^{k - 1} \chi_I (i) \stackrel{k \to \infty}{\to} 0 .$$
Let $(y_n)_{n = 0}^\infty$ be a sequence in $G$ such that $\left\{ n \in \mathbb{N}_0 : x_n \neq y_n \right\} \subseteq I$. Then
$$\frac{1}{k} \sum_{i = 0}^{k - 1} f (y_i) \stackrel{k \to \infty}{\to} \lambda .$$

In particular, if $(x_n)_{n = 0}^\infty$ is uniformly distributed, then $(y_n)_{n = 0}^\infty$ is uniformly distributed.
\end{Lem}

\begin{proof}
First, fix $f \in C(G)$, and suppose that $\frac{1}{k} \sum_{i = 0}^{k - 1} f (x_i) \stackrel{k \to \infty}{\to} \lambda$. Then
\begin{align*}
\frac{1}{k} \sum_{i = 0}^{k - 1} f(y_i)	& = \left( \frac{1}{k} \sum_{i = 0}^{k - 1} f(x_i) \right) + \left( \frac{1}{k} \sum_{i = 0}^{k - 1} \chi_{I}(i) (f(y_i) - f(x_i)) \right) \\
\Rightarrow \left| \left( \frac{1}{k} \sum_{i = 0}^{k - 1} f(y_i) \right) - \left( \frac{1}{k} \sum_{i = 0}^{k - 1} f(x_i) \right) \right|	& \leq \frac{1}{k} \sum_{i = 0}^{k - 1} \chi_{I} (i) \left| f(x_i) - f(y_i) \right| \\
	& \leq \frac{1}{k} \sum_{i = 0}^{k - 1} \chi_{I} (i) \left( 2 \left\| f \right\| \right) \\
	& \stackrel{k \to \infty}{\to} 0 .
\end{align*}

Now, suppose that $(x_n)_{n = 0}^\infty$ is uniformly distributed. Then the first part of this lemma tells us that for every $g \in C(G)$, we have that $\lim_{k \to \infty} \frac{1}{k} \sum_{i = 0}^{k - 1} g(y_i)$ exists and is equal to $\lim_{k \to \infty} \frac{1}{k} \sum_{i = 0}^{k - 1} g(x_i) = \int g \mathrm{d} \mu$.
\end{proof}

\begin{Lem}\label{Normal is dense}
Let $G$ be a compact abelian metizable group. If the set
$$F = \left\{ x \in G : \textrm{$\left( \Phi_n x \right)_{n = 0}^\infty$ is uniformly distributed} \right\}$$
is nonempty, and $\bigcup_{m = 1}^\infty \ker \Phi_m$ is dense in $G$, then $F$ is dense in $G$.
\end{Lem}

\begin{proof}
Let $U \subseteq G$ be a nonempty open subset, and let $a_0 \in F$. Then there exists $m \in \mathbb{N}$ and $p \in \bigcup_{m = 1}^\infty \ker \Phi_m$ such that $y \in U - p$, so $a_0 + p \in U$. Then $\Phi_n a_0 = \Phi_n (a_0 + p)$ for $n \geq m$, so it follows from Lemma \ref{Changes on a subset of density $0$} that $\left( \Phi_n (a_0 + p) \right)_{n = 0}^\infty$ is also uniformly distributed, i.e. $a_0 + p \in F \cap U$.
\end{proof}

\begin{Thm}\label{Stronger Mance Theorem}
Let $G = (G, \rho)$ be a compact abelian metrizable group with infinitely many elements and metric $\rho$. Let $(T_n)_{n = 1}^\infty$ be a sequence of continuous, surjective group endomorphisms $T_n : G \to G$. Let $\Phi_n = T_n T_{n - 1} \cdots T_1, \Phi_0 = \operatorname{id}_G$, and suppose that $\bigcup_{m = 1}^\infty \ker \Phi_m$ is dense in $G$. Suppose further that the set $F = \left\{ x \in G : \textrm{$\left( \Phi_n x \right)_{n = 0}^\infty$ is uniformly distributed} \right\}$ is nonempty. Then there exists $f \in C(G)$ such that the set of $x \in G$ such that $\lim_{k \to \infty} \frac{1}{k} \sum_{i = 0}^{k - 1} f \left( \Phi_i x \right)$ exists is meager.
\end{Thm}

\begin{proof}
Choose $\delta > 0$ such that $\mu \left( B(0, \delta) \right) < \frac{1}{8}$, which exists by Lemma \ref{Small balls}. Set
$$f(x) = \begin{cases}
	1	& \rho(x, 0) \leq \frac{\delta}{2} , \\
	2 - \frac{2}{\delta} \rho(x, 0)	& \frac{\delta}{2} \leq \rho(x, 0) \leq \delta , \\
	0	& \rho(x, 0) \geq \delta ,
\end{cases}$$ 

For each $K \in \mathbb{N}$, let $A_K$ be the set
$$A_K = \left\{ x \in G : \textrm{$\left| \left(\frac{1}{k_1} \sum_{i = 0}^{k_1 - 1} f \left( \Phi_i x \right) \right) - \left(\frac{1}{k_2} \sum_{i = 0}^{k_2 - 1} f \left( \Phi_i x \right) \right) \right| < \frac{1}{4}$ for all $k_1, k_2 \geq K$} \right\} .$$
Since the set of $x \in G$ such that $\lim_{k \to \infty} \frac{1}{k} \sum_{i = 0}^{k - 1} f \left( \Phi_i x \right)$ exists is contained in $\bigcup_{K = 1}^\infty A_K$, it will suffice to prove that each $A_K$ is nowhere dense. Now fix $K \in \mathbb{N}$.

Let $\mathcal{O}$ be a nonempty open subset of $G$. By Lemma \ref{Concentrated at zero}, there exists $L \geq K$ and a nonempty open subset $W_1 \subseteq \mathcal{O}$ such that $\frac{1}{L} \sum_{i = 0}^{L - 1} f \left( \Phi_i x \right) \geq \frac{7}{8}$. By Lemma \ref{Normal is dense}, there exists $a \in F \cap W_1$. Since $a \in F$, there exists $N \geq L$ such that
$$k \geq N \Rightarrow \left| \int f \mathrm{d} \mu - \frac{1}{k} \sum_{i = 0}^{k - 1} f \left( \Phi_i a \right) \right| \leq \frac{1}{8} .$$
But since $\int f \mathrm{d} \mu < \frac{1}{8}$, it follows that
$$\frac{1}{N} \sum_{i = 0}^{N - 1} f \left( \Phi_i a \right) < \frac{1}{4} .$$

Since $\frac{1}{N} \sum_{i = 0}^{N - 1} f \circ \Phi_i$ is uniformly continuous, it follows that there exists an open neighborhood $W_2$ of $a$ such that
$$x \in W_2 \Rightarrow \left| \left( \frac{1}{N} \sum_{i = 0}^{N - 1} f \left( \Phi_i x \right) \right) - \left( \frac{1}{N} \sum_{i = 0}^{N - 1} f \left( \Phi_i a \right) \right) \right| < \frac{1}{8} .$$
Then if $x \in W_2$, then
$$\frac{1}{N} \sum_{i = 0}^{N - 1} f \left( \Phi_i x \right) < \left( \frac{1}{N} \sum_{i = 0}^{N - 1} f \left( \Phi_i a \right) \right) + \frac{1}{8} < \frac{1}{4} + \frac{1}{8} = \frac{3}{8} .$$
Therefore, if $x \in W_1 \cap W_2$, then
$$\left( \frac{1}{L} \sum_{i = 0}^{L - 1} f \left( \Phi_i x \right) \right) - \left( \frac{1}{N} \sum_{i = 0}^{N - 1} f \left( \Phi_i x \right) \right) > \frac{7}{8} - \frac{3}{8} = \frac{1}{2}.$$
Thus $a \in W_1 \cap W_2 \subseteq \mathcal{O}$, and $\left( W_1 \cap W_2 \right) \cap A_K = \emptyset$.
\end{proof}

Throughout this section, we have relied heavily on the assumption that $\bigcup_{m = 1}^\infty \ker \Phi_m$ is dense in $G$. This assumption still encompasses a wide class of interesting examples when $G = \mathbb{R}^d / \mathbb{Z}^d, d \in \mathbb{N}$. Endow $G$ with the metric
$$\rho \left( \left( t_1, \ldots, t_d \right) + \mathbb{Z}^d , \left( s_1, \ldots, s_d \right) + \mathbb{Z}^d \right) = \sum_{j = 1}^d \min_{h \in \mathbb{Z}} \left| t_j - s_j + h \right| .$$
Note that the metric $\rho$ is invariant under addition by elements of $G$.

\begin{Lem}\label{Toral estimates}
Let $G = \mathbb{R}^d / \mathbb{Z}^d$, and let $T : G \to G$ be a continuous surjective group endomorphism. Let $A \in \mathbb{Z}^{d \times d}$ be the $d \times d$ integer matrix such that
$$T : \begin{bmatrix}
t_1 \\
\vdots \\
t_d
\end{bmatrix} + \mathbb{Z}^d \mapsto A \begin{bmatrix}
t_1 \\
\vdots \\
t_d
\end{bmatrix} + \mathbb{Z}^d .$$
Then for every $x \in G$ exists $y \in \ker T$ such that $\rho(x, y) \leq d^2 \left\| A^{-1} \right\|_\mathrm{op}$, where the operator norm is taken relative to the standard Euclidean norm on $\mathbb{R}^d$.
\end{Lem}

\begin{proof}
Let $\mathbf{e}_1, \ldots, \mathbf{e}_d$ be the standard basis of the real vector space $\mathbb{R}^d$, and let $\mathbf{f}_j = A^{-1} \mathbf{e}_j$ for $j = 1, \ldots, d$. Then
$$\ker T = \mathbb{Z} \mathbf{f}_1 + \cdots + \mathbb{Z} \mathbf{f}_d .$$
Let $x = (t_1, \ldots, t_d) + \mathbb{Z}^d \in G$. Since $A$ is invertible, we know that $\left\{ \mathbf{f}_1, \ldots, \mathbf{f}_d \right\}$ is a basis for $\mathbb{R}^d$, so there exist $\lambda_1, \ldots, \lambda_d \in \mathbb{R}$ such that $\sum_{j = 1}^d t_j \mathbf{e}_j = \sum_{j = 1}^d \lambda_j \mathbf{f}_j$. Let $\ell_j = \lfloor \lambda_j \rfloor$, and let $y = \sum_{j = 1}^d \ell_j \mathbf{f}_j + \mathbb{Z}^d$. Set $\kappa_j = \lambda_j - \ell_j \in [0, 1)$. Then $y \in \ker T$, and
\begin{align*}
\rho(x, y)	& = \rho \left(0, y - x \right) = \rho \left( 0 , \sum_{j = 1}^d \left( \lambda_j - \ell_j \right) \mathbf{f}_j + \mathbb{Z}^d\right) = \rho \left( 0 , \sum_{j = 1}^d \kappa_j \mathbf{f}_j + \mathbb{Z}^d \right) \\
& \leq \rho \left( 0 , \sum_{j = 1}^1 \kappa_j \mathbf{f}_j + \mathbb{Z}^d \right) + \rho \left( \sum_{j = 1}^1 \kappa_j \mathbf{f}_j + \mathbb{Z}^d , \sum_{j = 1}^2 \kappa_j \mathbf{f}_j + \mathbb{Z}^d \right) + \cdots + \rho \left( \sum_{j = 1}^{d - 1} \kappa_j \mathbf{f}_j + \mathbb{Z}^d , \sum_{j = 1}^d \kappa_j \mathbf{f}_j + \mathbb{Z}^d \right) \\
& = \sum_{j = 1}^d \rho \left( 0 , \kappa_j \mathbf{f}_j + \mathbb{Z}^d \right) .
\end{align*}
For each $j \in \{1, \ldots, d\}$, let $\mathbf{f}_j = \sum_{i = 1}^d b_{i, j} \mathbf{e}_j$. Then
\begin{align*}
\sum_{j = 1}^d \rho \left( 0 , \kappa_j \mathbf{f}_j + \mathbb{Z}^d \right)	& = \sum_{j = 1}^d \rho \left( 0, \kappa_j \sum_{i = 1}^d b_{i, j} \mathbf{e}_i + \mathbb{Z}^d \right) \\
	& = \sum_{j = 1}^d \sum_{i = 1}^d \min_{h \in \mathbb{Z} } \left| h - \kappa_j b_{i, j} \right| \\
	& \leq \sum_{j = 1}^d \sum_{i = 1}^d \kappa_j \left| b_{i, j} \right| \\
	& \leq \sum_{j = 1}^d \sum_{i = 1}^d \left| b_{i, j} \right|	& (0 \leq \kappa_j < 1) \\
	& \leq \sum_{j = 1}^d \sum_{i = 1}^d \left\| A^{-1} \mathbf{e}_j \right\| \\
	& \leq \sum_{j = 1}^d \sum_{i = 1}^d \left\| A^{-1} \right\|_\mathrm{op} \\
	& = d^2 \left\| A^{-1} \right\|_\mathrm{op} .
\end{align*}
\end{proof}

\begin{Prop}
Let $G = \mathbb{R}^d / \mathbb{Z}^d$, and let $(T_n)_{n = 1}^\infty$ be a sequence of continuous surjective group endomorphisms of $G$ onto itself. For each $n \in \mathbb{N}$, let $A_n \in \mathbb{Z}^{d \times d}$ be the $d \times d$ integer matrix such that
$$T_n : \begin{bmatrix}
	t_1 \\
	\vdots \\
	t_d
\end{bmatrix} + \mathbb{Z}^d \mapsto A_n \begin{bmatrix}
	t_1 \\
	\vdots \\
	t_d
\end{bmatrix} + \mathbb{Z}^d .$$
Then if $\liminf_{n \to \infty} \left\| \left( A_n A_{n - 1} \cdots A_1 \right)^{-1} \right\|_\mathrm{op} = 0$, then $\bigcup_{m = 1}^\infty \ker \Phi_m$ is dense in $G$.
\end{Prop}

\begin{proof}
Let $x \in G$, and let $\delta > 0$. Choose $m \in \mathbb{N}$ such that $\left\| \left( A_m A_{m - 1} \cdots A_1 \right)^{-1} \right\|_\mathrm{op} < \frac{\delta} {d^2}$. Then $\Phi_m$ is implemented by the matrix $A_m A_{m - 1} \cdots A_1$, so Lemma \ref{Toral estimates} tells us there exists $y \in \ker \Phi_m$ such that $\rho(x, y) \leq d^2 \left\| \left( A_m A_{m - 1} \cdots A_1 \right)^{-1} \right\|_\mathrm{op} < \delta$. Therefore $\bigcup_{m = 1}^\infty \ker \Phi_m$ is dense in $G$.
\end{proof}

	\bibliography{Bibliography}
\end{document}